\documentclass[11pt, reqno]{amsart}
\usepackage{amsmath, amsthm, amscd, amsfonts, amssymb, graphicx, color}
\usepackage[bookmarksnumbered, colorlinks, plainpages]{hyperref}
\usepackage{enumerate}

\usepackage[all]{xy}

\let\Item\item
\newenvironment{romanlist}{%

  \let\item\Item
  \begin{enumerate}
}{%
  \end{enumerate}}

\def\@begintheorem#1#2#3{\trivlist
   \item[\hskip \labelsep{\csname#1headfont\endcsname#1\ #2.} ]~\ignorespaces{\rm #3}\csname#1font\endcsname\ignorespaces}
\def\@opargbegintheorem#1#2#3{\trivlist
      \item[\hskip \labelsep{\csname#1headfont\endcsname#1\ #2\ }][#3].~\csname#1font\endcsname\ignorespaces}
\def\@endtheorem{\endtrivlist}

\newtheorem{theorem}{Theorem}

\newtheorem{lemma}{Lemma}
\newtheorem{proposition}{Proposition}
\newtheorem{definition}{Definition}[section]

\newtheorem{remark}{Remark}[section]

\begin{document}


\markboth{C. Ri et al.}{Entropy for A-coupled-expanding maps and
chaos}

\title{ENTROPY FOR A-COUPLED-EXPANDING MAPS AND CHAOS\\
 }

\author{CHOL-GYUN RI   and  HYON-HUI JU}
\address{CHOL-GYUN RI   and  HYON-HUI JU \newline 
Department of Mathematics, \\
 Kim Il Sung University, Pyongyang, D. P. R. Korea.}
\email{jhyonhui@yahoo.com}
\author{XIAOQUN WU}
\address{XIAOQUN WU \newline
School of Mathematics and Statistics, \\
Wuhan University, Hubei, 43072, China}
\email{xqwu@whu.edu.cn}

\maketitle

\begin{abstract}
The  concept of "$A$-coupled-expanding" map for a transition matrix
$A$ has been studied  as one of the most important criteria of chaos
in the past years. In this paper, the lower bound of the
topological entropy for strictly $A$-coupled-expanding maps  is
studied as a criterion for chaos in the sense of Li-Yorke, which is
less conservative and more generalized than the latest result is
 presented.  Furthermore, some conditions for $A$-coupled-expanding
 maps excluding the strictness   to be factors of subshifts of finite
 type  are derived. In addition, the topological entropy of  partition-$A$-coupled-expanding
map, which is put forward in this paper, is further estimated on
compact metric spaces.
 Particularly,  the topological entropy for partition-$A$-coupled-expanding circle maps
is given, with that for the Kasner map being calculated for
illustration and verification.
\end{abstract}



\section{Introduction}
\noindent The term "chaos" was first introduced into mathematics by
Li and Yorke \cite{LiYo 1975}. Since then, various definitions of
chaos have been proposed. However, in general, they do not coincide,
and none of them can be considered as a perfect definition of chaos.
Most of the definitions are based on the feature of long-term
unpredictability of chaotic behaviors due to sensitivity to initial
conditions.

The relations between those definitions have been widely studied, as
well as other measures of complexity, such as the topological
entropy,
 conjugacy or semi-conjugacy to symbolic dynamical systems,
  the Lyapunov exponent, the Hausdorff dimension, and so on.

The notion of "coupled-expansion" as a criterion for chaos was
originated from the terminology
 of turbulence \cite{BlCo 1986, BlCo 1992}
 and was considered as an important property of one-dimensional dynamical
 system. Specifically,
  a continuous map $f:I\rightarrow I$, where $I$ is the unit interval,
  is said to be \emph {turbulent} if there exist closed non degenerated subintervals
   $J$, $K$ of $I$  with pairwise disjoint interiors such that
    $J{\cup}K \subset f(J) \cap  f(K)$.

Furthermore, the map $f$ is said to be \emph {strictly turbulent} if
the subintervals $J$ and $K$ can be chosen disjoint. In fact, the
same concept was studied in one-dimensional dynamical system by
Misiurewicz \cite{Mis 1979, Mis 1980}, who called this property as
"horseshoe'', since it is similar to the Smale's horseshoe effect
\cite{Sma 1967}.  Let $f:I\rightarrow I$ be an interval map and
$J_1, {\cdots}, J_n$ be chosen non degenerate subintervals with
pairwise disjoint interiors such that
$J_1{\cup\cdots\cup}J_n{\subset}f(J_i)$ for $i=1, {\cdots}, n$. Then
$(J_1, {\cdots}, J_n)$ is called an \emph {$n$-horseshoe}, or simply
a \emph {horseshoe} if $n\geq2$.

In recently years, the study of chaos in the setting of general
topological dynamics has attracted wide attention.  In 2006, Shi and
Yu captured the essential meanings of the concept of turbulence for
continuous interval maps and extended it to maps in general metric
spaces \cite{ShYu 2006a}, where the maps were still called
turbulent. Since the term turbulence is well-established in fluid
mechanics, they changed the term "turbulence'' to the
"coupled-expansion''  \cite{ShYu 2006b}, which is more intuitive in
reflecting the conditions that the map satisfies. In 2009,  Shi,  Ju
and Chen extended the concept of "coupled-expansion'' to a more
general one -- "$A$-coupled-expansion'' for the transition matrix $A$
\cite{ShJuCh 2009}, which contains  "coupled-expansion'' as a
special case when each entry of the matrix $A$ equals to 1. In these
papers, several criteria of chaos induced by strictly
coupled-expanding maps or $A$-coupled-expanding maps  have been
established in  metric spaces, essentially with the compactness,  by
using the conjugacy to symbolic dynamical systems.
   Some  extended  criteria of chaos based on a result in \cite{ShJuCh 2009}
   in the  context of the complete metric space, can also be found in
\cite{XuYu 2010}.

Actually, the essential ideas of these notions include "horseshoe''
and "turbulence''. Thus, "Coupled-expansion'' and
"$A$-coupled-expansion''  in  metric spaces are based on domain
splitting  and  it is inclined to use conjugacy or semi-conjugacy to
symbolic dynamical systems, that is, shift or subshift of finite
type. This is one of the common and useful ways to study  chaos or
complex behaviors. Therefore, there are various other results, which
are more or less similar.

In 2001, Kennedy and York \cite{KeYo 2001} studied topological
horseshoes and proved that a continuous map in a compact invariant
set of a metric space could be topologically semi-conjugate to a
symbolic dynamical system under some hypotheses, called the
horseshoe hypotheses. Particularly, in 2001, Fu and Lu et al.
\cite{FuLuAsDu 2001} (Definition 7.1, Theorem 7.2) presented the
concept of "distillation'' and conditions for conjugacy to the
subshift of finite type, which are akin to the results in
\cite{ShJuCh 2009} but in the Hausdorff space .

As mentioned above, the $A$-coupled-expanding property has been
regarded as an  important criterion for chaos. In addition, the
positive topological entropy, which measures the complexity of a
system, is also known as one of the main indices to characterize
chaos. Therefore, it is natural to consider the topological entropy
for $A$-coupled-expanding maps and the relationships between various
concepts of chaos.  As to this topic, the topological entropy for a
$p$-horseshoe map (that is, a $p$-coupled expanding map) on the
interval was considered in \cite{Rue 2003}. In \cite{Miy 2002}, the
necessary and sufficient condition for a continuous map from a
circle into itself to have a positive topological entropy was
presented, and it was proved that the circle map is chaotic in the
sense of Devaney if and only if its topological entropy is positive.

In this paper, the topological entropy of  $A$-coupled-expanding
maps on compact metric spaces is studied. The lower bound of the
topological entropy for  strictly  $A$-coupled-expanding maps on
compact metric spaces is obtained and a less conservative and  more
generalized criterion for chaos in the sense of Li-Yorke is further
presented. A new concept, the partition-$A$-coupled-expanding map,
is defined
 excluding the usual strictness condition, and its topological  entropy is
estimated. Particularly, the topological entropy for the
partition-$A$-coupled-expanding circle map is given. As an
illustrative example,   the  topological entropy for the Kasner map
is calculated,  which is one of the most important tools in the
research on Cosmological models of Bianchi type in the big-bang
singular limit.

The paper is organized as follows. In Section \ref{pre},  some basic
concepts and lemmas are briefly introduced. In Section
\ref{entropy},  the topological entropy for the strictly
$A$-coupled-expanding maps on compact metric spaces is discussed and
a criterion for chaos in the sense of Li-Yorke which is less
conservative and more generalized than the result of \cite{ShJuCh
2009} is presented. In Section \ref{topentropy}, some conditions for
$A$-coupled-expanding maps to be a factor of subshifts of finite
type and the topological entropy for $A$-coupled-expanding circle
maps excluding the strictness condition are discussed. The
topological entropy and a chaotic property of the Kasner map are
further presented  as an illustrative example  of our theoretical
results for demonstration. A brief conclusion is drawn in Section
\ref{con}.

\section{Preliminaries}\label{pre}
\subsection{Basic Definitions}

A topological dynamical system is a pair  $(X, T)$, where  $X$ is a
compact metric space with a metric $d$ and $T$ is a surjective
continuous map from $X$ to itself.

Among the various definitions of chaos, the first and well-known one
appeared in mathematics  is chaos in the sense of Li and Yorke. It
is based on the idea of \cite{LiYo 1975} , however, was formalized
afterwards.

\begin{definition} A set $S\subset X$ is called a \emph
{scrambled set} if, for any two distinct points $x, y\subset S$,
\begin{displaymath}
\liminf\limits_{n\to\infty}d(T^n(x), T^n(y))=0
\quad\textrm{and}\quad \limsup\limits_{n\to\infty}d(T^n(x),
T^n(y))>0.
\end{displaymath}
The system $(X, T)$ is said to be \emph {chaotic in the sense of
Li-Yorke} if there exists an uncountable scrambled set $S$.
\end{definition}

\begin{definition}\cite{Rob 1995}
 A $p\times p$ matrix $A=(a_{ij})_{1\leq i, j\leq p}$ is said
to be \emph {transition} if $a_{ij}=0$ or 1 for all $i, j$;
$\sum\limits^p_{j=1}a_{ij}\geq 1$ for all $i$; and
$\sum\limits^p_{i=1}a_{ij}\geq 1$ for all $j$.
\end{definition}

\begin{definition} {\cite{BrSt 2003}} A
$p\times p$ matrix $A=(a_{ij})_{1\leq i, j\leq p}$ is said to be
\emph {nonnegative (positive)} if any entry of $A$ is nonnegative
(positive). The matrix $A$ is said to be \emph {irreducible} if for
any pair $(i, j), 1\leq i, j\leq p$, there exists a positive integer
$k$ such that $a^k_{ij}>0$, where $A^k=(a^k_{ij})_{1\leq i, j\leq
p}$. $A$ is said to be \emph {primitive} (or \emph{eventually
positive} \cite{ShJuCh 2009}) if $A^k$ is positive for some positive
integer $k$.
\end{definition}

\begin{definition} {\cite{ShJuCh 2009}} Let
$(X, d)$ be a metric space, $T:D\subset X\rightarrow X$ a map, and
$A=(a_{ij})_{1\leq i, j\leq p}$ a transition matrix, where $p\geq
2$. If there exist $p$  nonempty subsets $\Lambda_i(1\leq i\leq p)$
of $D$ with pairwise disjoint interiors such that
\begin{displaymath}
T(\Lambda_i)\supset\bigcup\limits_{\substack{j\\a_{ij}=1}}\Lambda_j,
1\leq i\leq p
\end{displaymath}
then $T$ is said to be \emph {$A$-coupled-expanding} in $\Lambda_i,
1\leq i\leq p$. Further, the map $T$ is said to be \emph {strictly
$A$-coupled-expanding in $\Lambda_i, 1\leq i\leq p$,} if $
d(\Lambda_i, \Lambda_j)>0$ for all $1\leq i\neq j\leq p$.
\end{definition}

\begin{remark}
 The dynamical system $(X, T)$ is said to be
$A$-coupled-expanding if $T$ is $A$-coupled-expanding.
\end{remark}

\begin{remark}
 In the special case that all entries of $A$ are
ones, the (strict) $A$-coupled-expansion  is (strict)
coupled-expansion or,  (strict) $p$-coupled-expansion.
\end{remark}

\begin{definition} \cite{PoYu 1998} Let $(X_1,T_1), (X_2, T_2)$ be
topological dynamical systems. If there exists a homeomorphism
$h:X_1\rightarrow X_2$ such that $h\circ T_1=T_2\circ h$, then $T_1$
is said to be \emph {topologically conjugate to $T_2$}. If $h$ is
continuous and surjective, but not necessarily invertible, and
$h\circ T_1=T_2\circ h$,   then $T_1$ is said to be \emph
{topologically semi-conjugate to $T_2$}, and $T_2$ is said to be a
\emph {factor of $T_1$}.
\end{definition}

\subsection{Topological Entropy}
\noindent  Adler, Konheim and McAndrew defined  topological entropy
for dynamical systems on compact metric spaces \cite{AdKoMc 1965}, a
key quantifier for complicated dynamical behaviors that plays an
important role in the
classification of dynamical systems. 
In the following the definition is briefly introduced. More details
can be found in \cite{PoYu 1998} and \cite{Wal 1982}.

Suppose $(X, T)$ be a topological dynamical system as addressed
above. Let $\alpha=\{A_i\}$ and $\beta=\{B_j\}$ be covers of $X$,
and define the \emph{refinement} as $\alpha\vee\beta=\{A_i\cap
B_j:A_i\cap B_j\neq \phi\}$. Moreover, if $\alpha^r=\{A^r_1,\cdots,
A^r_{N_r}\} (r=1, \cdots, k)$ are covers of $X$, define their
refinement as
\begin{displaymath}
\bigvee\limits_{r=1}^k\alpha^r=\{A^1_{i_1}\cap A ^2_{i_2}\cap
\cdots\cap  A^k_{i_k}\neq\phi: i_j\in \{1, \cdots, N_r\}, j=1,
\cdots, k\}.
\end{displaymath}

Denote $T^{-1}\alpha=\{T^{-1}A_1,\cdots,T^{-1}A_n\}$. More
generally, one has
\begin{eqnarray*}
\bigvee\limits_{i=0}^{k-1}T^{-i}\alpha&=&\alpha\vee\cdots\vee T^{-(k-1)}\alpha\\
&=&\{A_{i_0}\cap T^{-1}A_{i_1}\cap\cdots\cap
T^{-(k-1)}A_{i_{k-1}}\neq\phi:1\leq i_0,\cdots,i_{k-1}\leq n\}.
\end{eqnarray*}
Let $N(\alpha)$ be the smallest number of sets that can be used as a
subcover) of $\alpha$.

\begin{definition} The \emph{topological entropy for
$T$ relative to a cover $\alpha$} is defined by
\begin{displaymath}
h_{top}(\alpha, T)=\limsup\limits_{n\to\infty}\frac{1}{n}\log
N\left(\bigvee\limits_{i=0}^{n-1}T^{-i}\alpha\right).
\end{displaymath}
\end{definition}

\begin{remark} The sequence $\left(\frac{1}{n}\log
N\left(\bigvee\limits_{i=0}^{n-1}T^{-i}\alpha\right)\right)_{n\geq1}$
is sub-additive, therefore
\begin{displaymath}
h_{top}(\alpha, T)=\lim\limits_{n\to\infty}\frac{1}{n}\log
N\left(\bigvee\limits_{i=0}^{n-1}T^{-i}\alpha\right)=\inf\limits_{n\geq1}\frac{1}{n}\log
N\left(\bigvee\limits_{i=0}^{n-1}T^{-i}\alpha\right).
\end{displaymath}
\end{remark}

\begin{definition} The topological entropy for the
topological dynamical system $(X,T)$ is defined by
\begin{displaymath}
h_{top}(X, T)=\sup\{h_{top}(\alpha, T):\alpha \textrm{ is a finite
open cover of }X\}.
\end{displaymath}
\end{definition}

\vskip0.5cm 
\subsection{Transition matrix and graph}
 Define
the norm $\|\cdot\|$ of matrix $A=\left(a_{ij}\right)_{1\leq i,
j\leq p}$ by $\| A\|=\sum\limits_{1\leq i,j\leq p}|a_{ij}|$. It is
known that
$\limsup\limits_{n\to\infty}\|A^n\|^\frac{1}{n}=sup\{|\lambda|:\lambda
\textrm{ is the eigenvalue of }A\}$.

\begin{theorem}[Perron-Frobenius]
Let $A$ be a non-negative irreducible square matrix. Then there
exists an eigenvalue $\lambda$ of $A$ with the following properties:
(i) $\lambda>0$, (ii) $\lambda$ is a simple root of the
characteristic polynomial, (iii) $\lambda$ has a positive
eigenvector, (iv) if $\mu$ is any other eigenvalue of $A$, then
$|\mu|\leq\lambda$, (v) if $k$ is the number of eigenvalues of
modulus $|\lambda |$, then the spectrum of $A$(with multiplicity) is
invariant under the rotation of the complex plane by angle $2\pi/k$.
\end{theorem}

Let $A=(a_{ij})_{1\leq i,j\leq p}$ be a transition matrix and
$\Gamma_A$ be a directed graph associated to $A$. In other words,
$\Gamma_A$ is the directed graph with vertices $\{1,\cdots,p\}$ such
that there is an edge $i\rightarrow j$ if and only if $a_{ij}=1$.

\begin{definition} The eigenvalue $\lambda$
of a matrix is said to be \emph{maximal} if for any other eigenvalue
$\mu$,  $|\mu|\leq\lambda$.  The maximal eigenvalue of matrix $A$ is
denoted by $\lambda_A$. \end{definition}

\begin{definition} A cycle is said to be
\emph{full} in the graph if it has all vertices of the graph.
\end{definition}

\begin{definition} The graph is said to be
\emph{unified} if there is a full cycle in it.
\end{definition}

\begin{lemma}\label{lemma2p1}
A transition matrix $A$ is irreducible if and only if there is a
full cycle in $\Gamma_A$.
\end{lemma}
\begin{proof} It is obvious from the definition. \end{proof}

\begin{lemma}
Let $A=(a_{ij})_{1\le i,j\le p}$ and $B=(b_{ij})_{1\le i,j\le p}$ be
two transition matrices. If $A\le B$ then $\lambda_A\le \lambda_B$,
where $A\le B$ means $a_{ij}\le b_{ij}$ for $ 1\le i,j\le p$.
\end{lemma}
\begin{proof}  From Theorem 8.1.18 of \cite{HoJo 1985}, one gets $\rho (A) \le \rho (B)$,
 where $\rho (A)$ and $\rho (B)$ are respectively the spectrum radius of $A$ and $B$.
  Since $A$ and $B$ are nonnegative matrices, from  Theorem 8.3.1 in \cite{HoJo 1985},
  $\rho (A)$ and $\rho (B)$ are respectively the eigenvalue of $A$ and $B$.
  Therefore, $\lambda_A\le\lambda_B$. \end{proof}

\vskip0.5cm
\begin{lemma}
If $A$ is a transition matrix, then its maximal eigenvalue
$\lambda_A\ge1$.

\end{lemma}
\begin{proof}  This can be concluded directly from  Theorem 8.1.22 of \cite{HoJo 1985}.
 \end{proof}

\section{ Entropy for strictly A-coupled-expanding maps on  compact metric spaces and
chaos}\label{entropy}

\begin{lemma} \label{lemma3p1}
Let $(X, T)$ be a topological dynamical system and $A$ a transition
matrix. If the map $T:X\rightarrow X$ is strictly
$A$-coupled-expanding in $\Lambda_1,\cdots,\Lambda_p\subset X$, then
$T$ is also strictly $A$-coupled-expanding  in
$\overline{\Lambda_1},\cdots,\overline{\Lambda_p}$, where
$\overline{\cdot}$ is the closure of the set $\cdot$.
\end{lemma}
\begin{proof}
From the conditions, one obtains  $\overline{\Lambda_i}\cap
\overline{\Lambda_j}=\phi, 1\le i\neq j\le p$, and
$T(\Lambda_i)\supset\bigcup\limits_{\substack{j\\a_{ij}=1}}\Lambda_j,
1\leq i\leq p$. Furthermore, since $\overline{\Lambda_i}$ is compact
and $T$ is continuous, one has
\begin{displaymath}
T(\overline{\Lambda_i})\supset
\overline{\bigcup\limits_{\substack{j\\a_{ij}=1}}\Lambda_j}=
\bigcup\limits_{\substack{j\\a_{ij}=1}}\overline{\Lambda_j}, 1\leq
i\leq p.
\end{displaymath}This completes the proof.
\end{proof}

\begin{lemma}\label{lemma3p2}
Let $(X, T)$ be a topological dynamical system. If the map
$T:X\rightarrow X$ is strictly $p$-coupled-expanding in
$\Lambda_1,\cdots \Lambda_p\subset X$,
 then $h_{top}(X,T)\ge\log p$, where $p\ge 2$.
\end{lemma}
\begin{proof}
From Lemma \ref{lemma3p1}, one can suppose that $\Lambda_1,\cdots,
\Lambda_p$ are closed without loss of generality. From
$\Lambda_i\cap \Lambda_j\neq \phi,
 1\le i\neq j\le p$, there exist open sets $U_1,\cdots,U_p$ in $X$ such
 that $\Lambda_i\subset U_i, U_i\cap U_j=\phi, 1\le i\neq j\le p$.

Let $U_{p+1}=X\backslash\bigcup\limits_{i=1}^p\Lambda_i$. Then
$U_{p+1}$ is open and $\mathcal{U}=\{ U_1,\cdots,U_p,U_{p+1}\}$ is a
finite open
 cover of $X$,  where $U_{p+1}\cap\Lambda_i=\phi$ for $1\le i\le p$.

For any $k\ge 0$ and any $(i_0\cdots i_k)\in \{1,\cdots,p\}^{k+1}$,
define
\begin{displaymath}
\Lambda_{i_0\cdots i_k}=\Lambda_{i_0}\cap T^{-1}(\Lambda_{i_1})
\cap\cdots\cap T^{-k}(\Lambda_{i_k}).
\end{displaymath}
Therefore, $T^k(\Lambda_{i_0\cdots i_k})=\Lambda_{i_k}(1\le
i_0,\cdots,i_k\le p)$. This is obvious when $k=0$. While if $k>0$,
one has
\begin{eqnarray*}
T^k(\Lambda_{i_0\cdots i_k})&=&T\circ T^{k-1}(\Lambda_{i_0\cdots i_{k-1}}\cap T^{-k}(\Lambda_{i_k})) \\
&=&T\circ T^{k-1}\left(\Lambda_{i_0\cdots i_{k-1}}\cap T^{-(k-1)}\left(T^{-1}(\Lambda_{i_k})\right)\right) \\
&=&T\left(\Lambda_{i_{k-1}}\cap
T^{-1}(\Lambda_{i_k})\right)=T(\Lambda_{i_{k-1}})\cap
\Lambda_{i_k}=\Lambda _{i_k}.
\end{eqnarray*}
Therefore, $\Lambda_{i_0\cdots i_k}$ is non-empty and closed for any
$k (k\ge 0)$ and $(i_0\cdots i_k)\in \{1,\cdots,p\}^{k+1}$.

If $(i_0\cdots i_{n-1})\neq(j_0\cdots j_{n-1})$, there exists some
$k(0\le k\le n-1)$ such that $i_k\neq j_k$. Since
$T^k(\Lambda_{i_0\cdots i_{n-1}})\subset \Lambda_{i_k}\subset
U_{i_k}$, $T^k(\Lambda_{j_0\cdots j_{n-1}})\subset
\Lambda_{j_k}\subset U_{j_k}$, and $ U_{i_k}\cap U_{j_k}=\phi,$ the
set $U_{i_0}\cap T^{-1}(U_{i_1})\cap\cdots\cap
T^{-(n-1)}(U_{i_{n-1}})$ includes $\Lambda_{i_0\cdots i_{k-1}}$ but
$\Lambda_{j_0\cdots j_{k-1}}$. Thus, $U_{i_0}\cap
T^{-1}(U_{i_1})\cap\cdots\cap T^{-(n-1)}(U_{i_{n-1}})$ and
$U_{j_0}\cap T^{-1}(U_{j_1})\cap\cdots\cap T^{-(n-1)}(U_{j_{n-1}})$
are different. Therefore,
$N\left(\bigvee\limits_{i=0}^{n-1}T^{-i}\mathcal{U}\right) \ge p^n$.
Hence, $h_{top}(X, T)\ge h_{top}(\mathcal{U},
T)=\limsup\limits_{n\to\infty}\frac{1}{n}\log
N\left(\bigvee\limits_{i=0}^{n-1}T^{-i}\mathcal{U}\right)\ge \log
p$.
\end{proof}

\begin{lemma}\label{lemma3p3}
Let $(X, T)$ be a topological dynamical system and $A$ a transition
matrix. If the map $T:X\rightarrow X$ is strictly
$A$-coupled-expanding in $\Lambda_1,\cdots, \Lambda_p\subset X$,
then $h_{top}(X,T)\ge\log \lambda_A$.
\end{lemma}
\begin{proof}  From Lemma \ref{lemma3p1}, one can obtain that $\Lambda_1,\cdots, \Lambda_p$ are closed.
Let $\lambda_A$ be the maximal eigenvalue of some $A_k$, which is a
unified subgraph of $\Gamma_A$. Denote $A_k=(a_{ij})_{1\le i, j\le
m}$ and $A^n_k=(a^n_{ij})_{1\le i, j \le m}$. Then
\begin{displaymath}
\limsup\limits_{n\to\infty}\|A^n_k\|^{\frac{1}{n}}=\limsup\limits_{n\to\infty}\left(
\sum\limits_{1\le i, j \le
m}a^n_{ij}\right)^{\frac{1}{n}}=\lambda_A.
\end{displaymath}
Since
\begin{displaymath}
\limsup\limits_{n\to\infty}\left( \sum\limits_{1\le i, j \le
m}a^n_{ij}\right)^{\frac{1}{n}}= \max\limits_{1\le i, j \le
m}\limsup\limits_{n\to\infty}(a^n_{ij})^{\frac{1}{n}},
\end{displaymath}
there exist $i$ and $j$ $(1\le i, j \le m)$ such that
\begin{displaymath}
\limsup\limits_{n\to\infty}(a^n_{ij})^{\frac{1}{n}}=\lambda_A,
\end{displaymath}
where $a^n_{ij}$ means the number of paths with length $n$ from
vertex $i$ to vertex $j$ in the graph $\Gamma_{A_k}$. Since
$\Gamma_{A_k}$ is unified, there is an $n_0\ge 0$ such that
$a^{n+n_0}_{ii}\ge a^n_{ij}$, for all $n\in \mathbf{N}$. Thus,
$\limsup\limits_{n\to\infty}(a^{n+n_0}_{ii})^{\frac{1}{n}}\ge\lambda_A$,
where $a^{n+n_0}_{ii}$ means the number of cycles of length $n+n_0$
which includes vertex $i$. Denote these cycles by
\begin{displaymath}
\xymatrix{
             i \ar[r]             &            i^1_l \ar[r]          &    i^2_l \ar[d]  &  \\
i^{n+n_0-1}_l \ar[u]   &   i^{n+n_0-2}_l \ar[l]    &  \cdots ,\ar[l]
&  1\le l \le a^{n+n_0}_{ii}.     }
\end{displaymath}
Since $\Gamma_A$ is a graph on the set of vertices $\{1, \cdots , p
\}$ such that $i\rightarrow j$ if and only if
$T(\Lambda_i)\supset\Lambda_j$, define
\begin{displaymath}
C_l=\Lambda_i \cap T^{-1}\left(\Lambda_{i^{1}_l}\right)\cap
T^{-2}\left(\Lambda_{i^{2}_l}\right) \cap \cdots \cap
T^{-(n+n_0-1)}\left(\Lambda_{i^{n+n_0-1}_l}\right), 1\le l \le
a^{n+n_0}_{ii},
\end{displaymath}
then $C_l\neq\phi$, and if $l_1\neq l_2$ then $C_{l_1}\cap
C_{l_2}=\phi$. Furthermore,
$T^{n+n_0}\left(C_l\right)=\Lambda_i\supset
\bigcup\limits^{a^{n+n_0}_{ii}}_{l=1}C_l$ . Thus $T^{n+n_0}$ is
strictly $a^{n+n_0}_{ii}$-coupled-expanding. From Lemma
\ref{lemma3p2} $h_{top}(X,T^n)\ge\log a^{n+n_0}_{ii}$. Therefore,
$h_{top}(X,T)=\frac{1}{n}h_{top}(X,T^n)\ge\frac{1}{n}\log
a^{n+n_0}_{ii}$ for all $n\in\mathbf{N}$. Since
$\limsup\limits_{n\to\infty}(a^{n+n_0}_{ii})^{\frac{1}{n}}\ge\lambda_A$,
one has  $h_{top}(X,T)\ge\log\lambda_A$. \end{proof}

\begin{remark} It was recently proved that a continuous strictly
$A$-coupled-expanding map on compact subsets of a metric space is
topologically semi-conjugate to $\sigma _A$ (Theorem 3.1 of
\cite{XuYuGu 2011}). Lemma \ref{lemma3p3} can be  resulted
immediately from it using Proposition 4.1 of this paper.
\end{remark}

\begin{lemma}\label{lemma3p4}
Let $A=(a_{ij})_{1\le i, j \le p}$ be an irreducible transition
matrix. Then $\lambda_A>1$ if and only if there exists some
$i_0(1\le i_0 \le p)$ such
 that $\sum\limits^p_{j=1}a_{i_0 j} \ge 2$. In particular, $\lambda_A\ge 2^{\frac{1}{p}}$.
\end{lemma}
\begin{proof}
From Lemma \ref{lemma2p1}, it is known $\Gamma_A$ has a full cycle.
Firstly, suppose that there exists an $i_0(1\le i_0 \le p)$ such
that $\sum\limits^p_{j=1}a_{i_0 j} \ge 2$. Then vertex $i_0$ is
bifurcating, which means that it has 2 or more outgoing edges.
$\sum\limits^p_{j=1}a_{ij}^p$ means the number of $p$-paths from
vertex $i$, where $p$-paths are paths of length $p$. Since
$\Gamma_A$ has $p$ vertices and matrix $A$ is transition , for any
$i$,  at least one of the paths from vertex $i$ arrives at the
bifurcation vertex $i_0$ for at least $p-1$ times. Therefore, the
number of $p$-paths is no less than 2. Thus,
$\sum\limits^p_{j=1}a^p_{i j} \ge 2$.

A $2p$-path from any vertex $i$ can be considered as a $p$-path
which is from the end of a $p$-path passing the same vertex $i$.
Therefore, the number of $2p$-paths is no less than $2^2$. That is,
$\sum\limits^p_{j=1}a^{2p}_{ij} \ge 2^2$.

Similarly, $\sum\limits^p_{j=1}a^{np}_{ij} \ge 2^n$. Hence,
$\|A^{np}\| \ge 2^np$. Thus,
\begin{displaymath}
\lambda_A=\limsup\limits_{n\to\infty}\|A^{np}\|^{\frac{1}{np}} \ge
\limsup\limits_{n\to\infty}(2^np)^{\frac{1}{np}}=2^{\frac{1}{p}}.
\end{displaymath}

Finally,  suppose that $\sum\limits^p_{j=1}a_{i j}=1$ for any $ i$.
Then $\Gamma_A$ is a simple cycle passing all vertices, with each
vertex having only one outgoing edge and one incoming edge. So,
$\|A^n\|=\sum\limits^p_{j=1}a^n_{i j} = p$. Hence,
\begin{displaymath}
\lambda_A=\limsup\limits_{n\to\infty}\|A^{n}\|^{\frac{1}{n}}=1.
\end{displaymath}
This completes the proof.
\end{proof}

In \cite{BlGlKoMa 2002}, it was proved that if the topological
entropy for $(X,T)$ is positive, then there exists a scrambled
Cantor set. Particularly, $T$ is chaotic in the sense of Li-Yorke.
Here, by using Lemma \ref{lemma3p3} and Lemma \ref{lemma3p4}, one
obtains the following criterion for chaos.
\begin{theorem}\label{theorem3p1}
Let $(X, T)$ be a topological dynamical system and $A$ an
irreducible  matrix,
 which has a row with the row sum being no less than 2. If $(X,T)$ is strictly
 $A$-coupled-expanding, then $h_{top}(X,T)\ge \frac{\log 2}{p}$ and  $(X,T)$ is
 chaotic in the sense of Li-Yorke.
\end{theorem}

\begin{remark} As a criterion for chaos in the sense of Li-Yorke, the
conditions in this theorem are less conservative than that in
Theorem 5.1 in \cite{ShJuCh 2009}.
\end{remark}

\begin{remark}
In Theorem 3.1 of \cite{XuYuGu 2012}, it was proved that if a dynamical system on metric space is strictly A-coupled-expanding, then there is a subsystem which is topologically semi-conjugate to the finite symbolic dynamical system, furthermore, if the entropy of the finite symbolic dynamical system is positive then the topological dynamical system is chaotic in the sense of Li-Yorke.
They put emphasis on semi-conjugacy of the dynamical system to the finite symbolic dynamical system, while Theorem \ref{theorem3p1} of this paper puts emphasis on proof of the positivity of entropy for the topological dynamical system. The positivity of entropy cannot derived from Theorem 3.1 of \cite{XuYuGu 2012}.
\end{remark}


\section{Topological entropy for $A$-coupled-expanding circle maps and
chaos}\label{topentropy}

In this section we consider some conditions for
$A$-coupled-expanding maps to be factors of subshifts of finite
type, and the topological entropy for $A$-coupled-expanding circle
maps.
  To compute the upper bound of topological entropy for a map,
   the map must be considered as a full-system, not only as a sub-system. Unfortunately,
   the $A$-coupled-expanding map as a full-system generally does not satisfy the condition of strictness. For example,
    the Kasner map,  which is one of the important tools in the research on
    Cosmological models of Bianchi type in the big-bang singular limit, is a "partition''-$A$-coupled-expanding map excluding the strictness.
  In \cite{Rue 2003}, Ruette considered the lower bound of topological entropy
  for the p-coupled-expanding map from the interval to itself.
  In this section, we obtain  some conditions for $A$-coupled-expanding maps
   excluding the strictness to be factors of subshifts of finite type,
   as well as the upper and lower bounds of topological entropy for "partition''-$A$-coupled-expanding circle maps excluding the strictness.
\subsection{Some conditions for $A$-coupled-expanding maps to be factors of subshifts of finite type}

\begin{lemma}\label{lemma4p1}
Let $(X, T)$ be a topological dynamical system and $A$ a transition
matrix. If  $T$ is $A$-coupled-expanding in compact sets
 $\Lambda_1,\cdots, \Lambda_p\subset X$ such that
  $\bigcap\limits ^ \infty_{n=o} T^{-n}\left(\Lambda_{a_n}\right)$ is
   singleton for any $\alpha =(a_0 a_1 \cdots)\in \Sigma ^+_A$ , then
   there exists a nonempty compact invariant subset $\Lambda $ such
    that $\left(\Sigma^+_A, \sigma _A \right )$ is topologically semi-conjugate
     to $\left(\Lambda, T|_\Lambda \right)$, that is, the sub-system
       $\left(\Lambda, T|_\Lambda \right)$ is a factor of
         $\left(\Sigma^+_A, \sigma _A \right )$.
\end{lemma}

 \begin{proof} Let
\begin{displaymath}
\Lambda :=\bigcup\limits_{\alpha =(a_0 a_1 \cdots) \in \Sigma^+_A}
\left(\bigcap^\infty_{n=0} T^{-n} \left(\Lambda_{a_n}\right)\right).
\end{displaymath}
Obviously  $\Lambda \neq \emptyset$ . Define $\pi:\Sigma ^+_A
\rightarrow \Lambda$ as follows:
\begin{displaymath}
\pi(\alpha):=\bigcap^\infty_{n=0}T^{-n}\left(\Lambda_{a_n}\right),
\alpha=(a_0 a_1 \cdots)\in\Sigma^+_A.
\end{displaymath}
Then $\pi$ is a surjection  and $T(\Lambda)\subset \Lambda$. To show
that $\pi$ is continuous, choose any $\alpha=(a_0 a_1 \cdots)\in
\Sigma^+_A$. Since $\bigcap^\infty_{n=0} T^{-n}
\left(\Lambda_{a_n}\right)$ is singleton, one has
\begin{displaymath}
\forall \varepsilon\ge 0, \exists N_\varepsilon \in N:
d\left(\bigcap
^{N_\varepsilon}_{n=0}T^{-n}\left(\Lambda_{a_n}\right)\right)\le\varepsilon,
\end{displaymath}
where $d(\bullet)$ denotes the diameter of $\bullet$. Let
$\delta=\frac{1}{2^{N_\varepsilon}+1}$, then
\begin{displaymath}
 \beta \in \Sigma ^+_A,d_{\Sigma^+_A}\left(\alpha, \beta\right)\le \delta\Rightarrow a_n =b_n, 0\leq n\leq N_\varepsilon.
\end{displaymath}
Therefore $d(\pi(\alpha),\pi( \beta))\le\varepsilon$, that is, $\pi$
is continuous. Thus, from the compactness of $\Sigma^+_A$ and
continuity of $\pi$, one obtains that $\Lambda
=\pi\left(\Sigma^+_A\right)$ is compact.

Next, for any $\alpha=(a_0 a_1 \cdots )\in \Sigma^+_A$,
\begin{displaymath}
\pi\circ\sigma_A(\alpha)=\pi(\sigma_A(a_0 a_1\cdots))=\pi(a_1
a_2\cdots)=\bigcap^\infty_{n=0}T^{-n}
\left(\Lambda_{a_n}\right)=T\circ\pi(\alpha),
\end{displaymath}
 that is, $\pi\circ \sigma_A=T \circ \pi$ holds.
The proof is thus completed.
\end{proof}

\begin{remark} Lemma \ref{lemma4p1} presents a condition for a subsystem of a given
system to be a factor of a subshift of finite type. However,  the
condition for full systems themselves to be factors of subshifts of
finite type requires estimating the upper bound of the topological
entropy for the systems, which is discussed in the following
theorems.
\end{remark}
\begin{theorem} \label{thm4p1}
Let $(X,T)$ be a topological dynamical system and $A$ a transition
matrix. Let $T$ be $A$-coupled expanding in sets $\Lambda_1, \cdots,
\Lambda_p \subset X$,
 where $\bigcup^p_{j=1}\Lambda_j=X$ (it is called partition-$A$-coupled-expanding). If $T$
 satisfies the following conditions:
 \begin{romanlist}[(iii)]
\item  For any $\alpha =(a_0 a_1\cdots)\in \Sigma^+_A$,
$\bigcap\limits ^ \infty_{n=o}
 T^{-n}\left(\Lambda_{a_n}\right)$ is a singleton,
 \label{th4p1item1}
\item
$T(\Lambda_i)=\bigcup\limits_{\substack{j\\a_{ij}=1}}\Lambda_j,
1\leq i\leq p$,  \label{th4p1item2}
\end{romanlist}
 then  $\left(\Sigma^+_A, \sigma _A \right )$ is topologically semi-conjugate
  to $(X,T)$, that is, the  full system $(X,T)$ is a factor of
   $\left(\Sigma^+_A, \sigma _A \right )$.
\end{theorem}

\begin{proof}
It is sufficient to prove that the set
\begin{displaymath}
\Lambda =\bigcup_{\alpha =(a_0 a_1 \cdots)\in \Sigma^+_A} \left(
\bigcap^\infty_{n=0}T^{-n}\left(\Lambda_{a_n}\right) \right)
\end{displaymath}
considered in Lemma \ref{lemma4p1} is equal to the set $X$. In fact,
condition (\ref{th4p1item2}) implies that
\begin{displaymath}
\forall x\in X, \exists \alpha =(a_0 a_1 \cdots)\in\Sigma^+_A:
T^n(x)\in\Lambda_{a_n}.
\end{displaymath}
That is, $X\subset \Lambda$. Therefore, one has  $\Lambda =X$.
\end{proof}

\begin{remark} Condition (\ref{th4p1item2})  in Theorem \ref{thm4p1} is the
necessary and sufficient condition for $X$ to be equal to
\begin{displaymath}
\Lambda =\bigcup_{\alpha =(a_0 a_1 \cdots)\in \Sigma^+_A} \left(
\bigcap^\infty_{n=0}T^{-n}\left(\Lambda_{a_n}\right) \right).
\end{displaymath}
\end{remark}

\begin{lemma}\label{lemma4p2}
Let  $A$ be a transition matrix, $(X,T)$ be a topological dynamical
system, and $T$
 be $A$-coupled-expanding in compact sets $\Lambda_1, \cdots, \Lambda_p\subset X$.
 If for any $\alpha=(a_0 a_1\cdots)\in \Sigma^+_A$,  a set $\bigcap^\infty_{n=0} \overline{T^{-n}
 \left(int\Lambda_{a_n}\right)} $ is singleton, then there exists a nonempty
 compact invariant set $\Lambda$ such that $(\Sigma ^+_A, \sigma_A)$ is topologically
  semi-conjugate to $(\Lambda, T|_\Lambda)$  (where $int  \Lambda_{a_n}$ represents the interior of $ \Lambda_{a_n}$
).
\end{lemma}

\begin{proof} Let
\begin{displaymath}
\Lambda=\bigcup_{\alpha=(a_0
a_1\cdots)\in\Sigma^+_A}\left(\bigcap^\infty_{n=0}
\overline{T^{-n}\left( int\Lambda_{a_n}\right)}\right)
\end{displaymath}
and
\begin{displaymath}
\pi(\alpha)=\bigcap^\infty_{n=0}\overline{T^{-n}\left(int\Lambda_{a_n}\right)},
\alpha=(a_0 a_1\cdots)\in\Sigma^+_A.
\end{displaymath}
This lemma can be proved in the same way as that of Lemma
\ref{lemma4p1}.
\end{proof}

\begin{remark}   Lemma \ref{lemma4p2} can be conveniently applied to  various systems in
differentiable spaces.
\end{remark}

\begin{remark}  The condition of Lemma \ref{lemma4p2} is weaker than that of
Lemma \ref{lemma4p1}. In fact,
\begin{displaymath}
\bigcap^\infty_{n=0} T^{-n}\left(\Lambda_{a_n}\right)\supset
\bigcap^\infty_{n=0} \overline{T^{-n}\left( int
\Lambda_{a_n}\right)}.
\end{displaymath}
However, the inverse does not hold in general.
\end{remark}

\begin{theorem}\label{thm4p2}
Let  $A$ be a transition matrix and $(X,T)$  a topological dynamical
system. Suppose that $T$ is partition-$A$-coupled-expanding in sets
$\Lambda_1, \cdots, \Lambda_p\subset X$
 such that $\bigcup^p_{i=1} \overline{int \Lambda_i}=X$. If  $T$ satisfies
 the following conditions:
 \begin{romanlist}[(iii)]
\item For any $\alpha =(a_0 a_1\cdots)\in\Sigma^+_A ,$ $
\bigcap^\infty_{n=0} \overline{T^{-n}\left( int
\Lambda_{a_n}\right)}$ is singleton, \label{th4p2item1}

\item $\overline{int
\Lambda_i}\subset\bigcup\limits_{\substack{j\\a_{ij}=1}}
\overline{T^{-1}\left( int\Lambda_j\right)}$, \label{th4p2item2}
 \end{romanlist}
then $(\Sigma^+_A, \sigma_A)$ is topologically semi-cojugate to
$(X,T)$, that is, the full system  $(X,T)$ is a factor of
$(\Sigma^+_A, \sigma_A)$.
\end{theorem}

\begin{proof} Denote
\begin{displaymath}
\Lambda=\bigcup_{\alpha=(a_0
a_1\cdots)\in\Sigma^+_A}\left(\bigcap^\infty_{n=0}
\overline{T^{-n}\left( int\Lambda_{a_n}\right)}\right).
\end{displaymath}
Since
\begin{displaymath}
\bigcup^p_{i=1}\overline{int\Lambda_i}=X,
\end{displaymath}
it follows that
\begin{displaymath}
\forall x\in X, \exists a_0\in \{1,\cdots,p\}: x\in
\overline{int\Lambda_{a_0}}.
\end{displaymath}
Next, from  condition (\ref{th4p2item2}) and by reduction, one can
conclude that
\begin{displaymath}
\exists a_n \in \{ 1,\cdots,p\}((A)_{a_{n-1,n}}=1): x\in T^{-n}(int
\Lambda_{a_n}),  n\geq1.
\end{displaymath}
Let $\alpha=(a_0 a_1\cdots)$, it is obvious that
$\alpha\in\Sigma^+_A$ and $x\in \Lambda$. Therefore $\Lambda=X$.
This completes the proof.
\end{proof}

\subsection{Entropy for partition-$A$-coupled-expanding circle maps}
 In this subsection,the topological entropy for
partition-$A$-coupled-expanding circle maps is investigated based on
 above results.

\begin{proposition} \label{prop4p1}
\cite{PoYu 1998} Let $(X_1, T_1)$ and $(X_2, T_2)$ be topological
dynamical systems. If  system $(X_1, T_1)$ is topologically
 semi-conjugate to system $(X_2 ,T_2)$,  then $h_{top}(X_1, T_1)\geq h_{top}(X_2,
 T_2)$.
\end{proposition}

\begin{theorem}\label{thm4p3}
Let $X$ be an interval or a circle, $(X,T)$ a topological dynamical
system and $A$ a transition matrix. If $T$ is a
partition-$A$-coupled-expanding map in sets $\Lambda_1, \cdots,
\Lambda_p \subset X$, satisfying the following conditions:
 \begin{romanlist}[(iii)]
\item For any $\alpha =(a_0 a_1\cdots)\in \Sigma^+_A, \bigcap\limits ^ \infty_{n=o} T^{-n}\left(\Lambda_{a_n}\right)$ is
singleton, \label{thm4p3item1}
\item  $T(\Lambda_i)=\bigcup\limits_{\substack{j\\a_{ij}=1}}\Lambda_j, 1\leq i\leq p$,\label{thm4p3item2}
 \end{romanlist}
then
\begin{displaymath}
 h_{top}(X,T)=\log \lambda_A.
\end{displaymath}
\end{theorem}

\begin{proof}
On one hand, since $(\Sigma^+_A, \sigma_A)$ is topologically
semi-conjugate to $(X,T)$, as can be obtained from Theorem
\ref{thm4p1}, one has
\begin{displaymath}
h_{top}(X,T)\leq h_{top}(\Sigma^+_A, \sigma_A)=\log \lambda_A
\end{displaymath}
from Proposition \ref{prop4p1}.

 On the other hand, by using  the method similar
to that used in the proof of Proposition 4.2.16 of \cite{Rue 2003},
for an interval map, one can conclude that
\begin{displaymath}
h_{top}(X,T)\geq \log \lambda_A.
\end{displaymath}
Therefore, $h_{top}(X,T)=\log \lambda_A.$ This completes the proof.
\end{proof}

\begin{theorem}\label{thm4p4}
Let $S_1$ be a unit circle, $A$ a transition matrix and
$T:S^1\rightarrow S^1$ a partition-$A$-coupled-expanding map in
closed arcs of the
 circles $\Lambda _1, \cdots, \Lambda_p\subset S^1$. If $T$ satisfies
 following conditions:
  \begin{romanlist}[(iii)]
\item  For some $r> 1$  and any arc $V\subset\Lambda_i$,
 $d(T(V))\geq rd(V), 1\leq i\leq p$, where $d(\bullet)$ denotes the natural length of $\bullet$,
\label{thm4p4con1}
\item $\bigcup\limits_i T(\partial \Lambda_i)\subset \bigcup\limits _i
\partial \Lambda_i$, \label{thm4p4con2}
\end{romanlist}
then $(S^1,T)$ is a factor of  $(\Sigma^+_A, \sigma_A)$,   and
$h_{top}(S^1,T)=\log \lambda_A.$
\end{theorem}

\begin{proof}
 For $N\geq 0$ define a  set $D_N$ as
\begin{displaymath}
D_N=\bigcap ^N_{n=0}T^{-n}\left(\Lambda_{a_n}\right).
\end{displaymath}
It is obvious that $D_0\supset D_1\supset\cdots$. Since $T$ is continuous
and $\Lambda_{a_n}$ is compact for any $n$, $D_N\ne\emptyset$ is
also compact. Therefore, for any $N$
\begin{displaymath}
\bigcap^N_{n=0}D_n\ne\emptyset.
\end{displaymath}
From  condition (i),

\begin{displaymath}
T(\Lambda_{a_0}\cap T^{-1}(\Lambda_{a_1}))=\Lambda_{a_1},
\end{displaymath}
\begin{displaymath}
d(\Lambda_{a_0}\cap T^{-1}(\Lambda_{a_1}))\leq
\frac{1}{r}d(\Lambda_{a_i})\leq \frac{1}{r}d(S^1), i = 1, 2
\end{displaymath}
\begin{displaymath}
d(D_N)\leq \frac{1}{r^N}d(S^1).
\end{displaymath}
It follows that  $d(D_N)\rightarrow 0 (N\rightarrow \infty)$
 and $\bigcap^\infty_{n=0}T^{-n} (\Lambda_{a_n})$ consists of one point, which
  satisfies  Condition (\ref{thm4p3item1}) of Theorem \ref{thm4p3}. Furthermore, Condition (\ref{thm4p4con2}) of this theorem satisfies
   that  of Theorem \ref{thm4p3}. Therefore, the conclusion can be drawn. Thus the proof is
   completed.
\end{proof}

In \cite{Miy 2002}, it was proved that Devaney's chaos is equivalent
to having positive entropy for a continuous circle map. Therefore,
one can get the following criterion as a sufficient condition for a
circle map to be chaotic based on Theorem \ref{thm4p3} ,
\ref{thm4p4} and Lemma \ref{lemma3p4}.

\begin{proposition}\label{prop4p2}
Let $A$ be an irreducible matrix, which has a row with the row-sum
no less than 2. If $T:S_1\rightarrow S_1$ is a continuous
  $A$-coupled-expanding map satisfying the assumptions of Theorem
  \ref{thm4p3} or Theorem \ref{thm4p4}, then $T$ is chaotic in the sense of Devaney as well as Li-Yorke.
\end{proposition}

\subsection{A numerical example}

\subsubsection{Background of the Kasner map in Cosmological models of Bianchi type $\bf{IX}$}

Cosmological models of Bianchi type yield spatially homogeneous,
anisotropic solutions $g_{\alpha \beta}$
  of the Einstein field equations,

$$
 R_{\alpha \beta}-\frac {1}{2}Rg_{\alpha \beta}=T_{\alpha \beta}.
$$
Here $R_{\alpha \beta}$ denotes the Ricci curvature tensor and $R$
the scalar curvature of the Lorenzian metric $g_{\alpha \beta}$,
whereas $T_{\alpha \beta}$ denotes the stress energy tensor.

This problem can be reduced to a five-dimensional system of ordinary
differential equations in expansion-normalized variables
representing the spatial homogeneity by a three-dimensional Lie
algebra. For unimodal Lie algebras, Bianchi class A, the reduced
equations  are \cite{Ren 1997} \cite{Rin 2001}
\begin{eqnarray}
N_1'&=&(q-4\Sigma _+)N_1,     \nonumber  \\
N_2'&=&(q+2\Sigma_+ +2\sqrt{3}\Sigma_-)N_2,\nonumber \\
N_3'&=&(q+2\Sigma_+ -2\sqrt{3}\Sigma_-)N_3,\\
\Sigma_+'&=&-(2-q)\Sigma_+ -3S_+,\nonumber \\
\Sigma_-'&=&-(2-q)\Sigma_- -3S_-,\nonumber
\end{eqnarray}
where
\begin{eqnarray}
q&=&\frac {1}{2}(3\gamma -2)\Omega +2(\Sigma _-^2 +\Sigma _+^2),\nonumber \\
S_+&=&\frac{1}{2}[(N_2-N_3)^2-N_1(2N_1-N_2-N_3)],\\
S_-&=&\sqrt{3}(N_3-N_2)(N_1-N_2-N_3).\nonumber
\end{eqnarray}
Here, the superscript $'$ denotes the derivative with respect to
time
 $\tau$,  $N_i (i=1,2,3)$ are the spatial curvature variables,
 $\Sigma _+ $ and  $\Sigma _- $ are the shear variables, $q$ is the deceleration
parameter,  $\Omega$ is the density parameter, and $\gamma (\frac
{2}{3}<\gamma \leq 2)$ describes the uniformly distributed matter.
The Hamiltonian constraint is
\begin{eqnarray}
\Omega+\Sigma_+^2+\Sigma_-^2+\frac{3}{4}(N_1^2+N_2^2+N_3^2-2(N_1N_1+N_2N_2N_3+N_3N_1))=1.
\end{eqnarray}

 The Bianchi type to which a solution of (1) corresponds depends on
 the the values of $N_1$, $N_2$ and $N_3$. If all the three are zeros the
 Bianchi type is ${\bf I}$. If precisely one is non-zero then it is type ${\bf II}$. If
 precisely two are nonzero it is either type ${\bf VI_0}$(signs  opposite)
 or type ${\bf VII_0}$(signs  equal). If all three are non-zero it is
 either  type ${\bf IX}$(all signs  equal) or type ${\bf VIII}$(one sign
 different from the other two).

 The invariant set  of (1) with $\Omega =0$ corresponds to the
vacuum model. The set of equilibria of Bianchi type ${\bf I}$
becomes $\{(N_1,N_2,N_3,\Sigma_+,\Sigma _-,\Omega): N_1=N_2=N_3=0,
\hspace{0.5cm} \Sigma_+^2+\Sigma _-^2=1,  \Omega =0\}$. The set
$K=\{(\Sigma_+,\Sigma _-):\Sigma_+^2+\Sigma _-^2=1\}$ on the $\Sigma
_\pm $-plane is called the \emph{Kasner circle}. There are three
special points on the Kasner circle $K$ with coordinates
$(-1,0),(\frac {1}{2},\pm \frac {\sqrt {3}}{2})$, which divide the
circle $K$ into three equal parts. They are denoted by $T_1$, $T_2$
and $T_3$ \cite{Ren 1997}, as shown Fig.1. The $\alpha$-limit set of
a solution of type ${\bf II}$ with $N_1\not=0$ lies on the longer
one of the two open arcs with endpoints $T_2$ and $T_3$, while the
$\omega$-limit set lies on the shorter one of the  arcs. The points
$T_1$, $T_2$ and $T_3$ are permuted cyclically by the threefold
symmetry, which shows what happens for type ${\bf II}$ solution with
$N_2\not=0$ or $N_3\not=0$. In other words, the trajectories of
Bianchi type ${\bf II}$ vacuum solutions consist of heteroclinic
orbits to equilibria on the Kasner circle, of which projections onto
the $\Sigma _\pm $-plane yield straight lines through the point
$(\Sigma _+,\Sigma _-)=(2,0)$ in the case of $N_1\not= 0$. The
projections of the other cases of Bianchi type ${\bf II}$($N_2\not
=0$ or $N_3\not=0$)  are given similarly.

Therefore, if $x$ is a point in $K\setminus \{T_1,T_2,T_3\}$, then
there is a point $y$ in $K$ and vacuum type ${\bf II}$ heteroclinic
orbit with $x$ being an $\omega$-limit point and $y$ being an
$\alpha$-limit point. Furthermore,  if $x$ lies on the shorter of
the two open arcs with endpoints $T_2$ and $T_3$, then $y$ lies on
the longer of these arcs while $x$ and $y$ lie on the straight lines
through the point $(\Sigma _+,\Sigma _-)=(2,0)$ in the case of
$N_1\not= 0$. The same results can be obtained for the other two
cases, $N_2\not =0$ or $N_3\not=0$, whereas the straight line goes
through the point $(-1,-\sqrt {3})$ and $(-1,+\sqrt {3})$
respectively.

The \emph{Kasner map} $\varphi:K\rightarrow K$  maps each $x$ in
$K\setminus \{T_1,T_2,T_3\}$ to $y$ in $K$, and maps $T_i$ to $T_i$
$(i=1,2,3)$.

  The $\alpha$-limit of  system (1) corresponds to the initial singularity
  (big-bang singularity) of the cosmological model.  Belinskii, Khalatnikov
and Lifshit \cite{Belinskii 1982} and Misner \cite{Misner}
conjectured that the dynamics of the Bianchi type ${\bf IX}$ type
models in this limit follows the Kasner map. In \cite{Rin 2001}, it
was proved that at least for Bianchi type ${\bf IX}$ solutions, the
Bianchi attractor formed by the union of the Kasner circle and its
heteroclinic orbits is indeed an attractor for trajectories to
generic initial data under the time-reversed flow.

\subsubsection{Topological Entropy for  the Kasner map}

Here we will consider the entropy and chaotic property of the Kasner
map  using the above results.

The chaotic dynamics of the  Bianchi type ${\bf IX}$ cosmological
models in the big-bang singular limit has been widely discussed in
the last few decades. The transient behavior of the ${\bf IX}$
models towards the initial singularity can be described by sequences
of anisotropic Kasner states, that is, Bianchi type  ${\bf I}$
vacuum solutions. These sequences are determined by a discrete map,
Kasner map, which implies an oscillatory anisotropic behavior. In
the first work by Barrow \cite{Barr 1982}, the chaotic property of
this discrete map represented by the Gauss map, was studied.  But
the Gauss map itself corresponds to a specific time slicing and the
ambiguity of time was manifest in this research area  \cite{Ren
1997}. However, the Kasner map represented by the above mentioned
form as shown in Fig. 1 has been discussed, which doesn't depend on
any time slicing  \cite{Ren 1997}. Therefore, it is meaningful to
consider the chaotic property for this map, especially in the
circumstance of focus on the research interest in chaos for
cosmological models  in the big-bang singular limit.

 We will firstly show that this type of the Kasner map is $A$-coupled-expanding
 for a primitive matrix $A$ excluding the strictness and compute its topological entropy,
 and discuss its chaotic properties  in the sense of Li-Yorke and Devaney.

For convenience, the polar co-ordinate system $(r,\theta)$ is
introduced in the $(\Sigma_+, \Sigma_-)$-plane, thus the Kasner
circle $K$ can be denoted as
\begin{equation*}
 K=\{(1,\theta)|\hspace{0.5cm}\theta\in[0,2\pi]\}.
\end{equation*}
The co-ordinates of the special points $T_1$, $T_2$, $T_3$ are
respectively $(1,\pi/3)$, $(1,5\pi/3)$ ,  and  $(1,\pi)$. Let
$\Lambda _1$ be the shorter of the arcs between $T_2$ and $T_3$ on
the circle, $\Lambda_2$ and $\Lambda_3$ be the shorter of the arcs
between $T_1$ and $T_3$,  $T_1$ and $T_2$,respectively. In
particular,
\begin{eqnarray*}
\Lambda_1&=&\bigg\{(1,\theta) |\hspace{0.5cm}\theta\in [0, \frac
{\pi}{3}] \cup [\frac
{5\pi}{3},2\pi]\bigg\},\\
\Lambda_2&=&\bigg\{(1,\theta)|\hspace{0.5cm}\theta\in[\pi, \frac{5\pi}{3}]\bigg\},\\
\Lambda_3&=&\bigg\{(1,\theta)|\hspace{0.5cm}\theta\in[\frac{\pi}{3},\pi]\bigg\},
\end{eqnarray*}
as shown in Fig.~\ref{fig1}.

\begin{figure}\label{fig1}
\centering
\includegraphics[width=8cm]{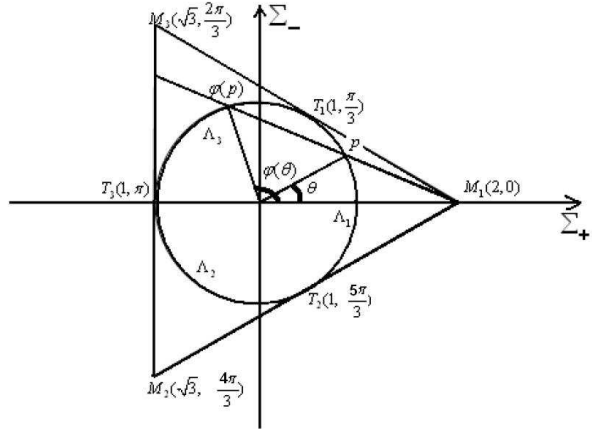}
\caption{The Kasner map $\varphi $ and the Kasner circle.}
\end{figure}

For an eventually positive $3\times 3$ matrix
\begin{equation*}
A_0=\left(\begin{array}{ccc}0&1&1\\1&0&1\\1&1&0\end{array}\right),
\end{equation*} one  can see that the Kasner map $\varphi$ is an
$A_0$-coupled-expanding map in $\Lambda_1$, $\Lambda_2$ and
$\Lambda_3$.

In fact,  it is easy to see that $\Lambda _1$, $\Lambda_2$ and
$\Lambda_3$ are closed subsets with disjoint interiors in $K$ and
satisfy
\begin{eqnarray*}
\varphi (\Lambda _1)&=&\Lambda _2\cup \Lambda _3,\\
\varphi (\Lambda _2)&=&\Lambda _3\cup \Lambda _1,\\
\varphi (\Lambda _3)&=&\Lambda _1\cup \Lambda _2.
\end {eqnarray*}

For convenience, one can denote points on $K$ by one parameter
$\theta \in [0,2\pi]$/(mod $2\pi$)  and define the natural metric on
the Kasner circle $K$ as
\begin{equation*}
d(\theta_1,\theta_2)=min\{|\theta_1-\theta_2|,\hspace{0.3cm}
2\pi-|\theta_1-\theta_2| \}, \hspace{0.5cm} \theta_1,\theta_2
\in[0,2\pi) \end{equation*}

 The Kasner map $\varphi:K\rightarrow K$ can be expressed by a map
$\Phi:[0,2\pi]/\sim$ $\rightarrow$ $[0,2\pi]/\sim$, where $\sim$
indicates that 0 and $2\pi$ are identical. Then one has the
following lemma.

\begin{lemma}\label{lemma4p3}
The map $\Phi$ satisfies
\begin{romanlist}[(iii)]
\item $\Phi(\theta)\in C^1([0,2\pi]/\sim)$,
\item  $|\Phi'(\theta)|\geq1$ for any $\theta \in K$. \label{lemma4p3item2}
\end{romanlist}
 The equal sign of (\ref{lemma4p3item2}) holds only for the
special points $T_1$, $T_2$ and $T_3$, that is,  for $
\theta=\pi,\pi/3$ and $ 5\pi/3$.
\end{lemma}

 It is sufficient to consider for $\theta \in [0,\pi/3]$ by symmetry of the Kasner map.
 This lemma can thus be easily proved by using the fact that the map  $\Phi$ for
 $\theta \in [0, \pi/3]$  can be described by

\begin{equation}\label{eq4}
\Phi(\theta)=\pi-\theta-2\arctan\frac{\sin\theta}{2-\cos\theta}:
\theta \in[0,\frac{\pi}{3}].
\end{equation}

\begin{proposition}\label{prop4p3}
The Kasner map is chaotic in the sense of Devaney as well as
Li-Yorke. Moreover the Kasner map  is a factor of the subshift
$\sigma _{A_0}$ and its topological entropy is $\log 2$.
\end{proposition}
\begin{proof}
 $A_0$ is an irreducible matrix  with row-sum 2, and it is obvious that  $\Phi$ is
  a partition-$A_0$-coupled expanding map satisfying Condition (\ref{thm4p3item2}) of
   Theorem \ref{thm4p3}.
  From (\ref{eq4}), it can be confirmed that condition 1 of Theorem \ref{thm4p3} is also satisfied.
   Therefore,  the Kasner map is chaotic in the sense of Devaney as well as  Li-Yorke
   from Proposition \ref{prop4p2}. From Theorem \ref{thm4p3}, one can see that the second statement obviously
   holds  since the largest eigenvalue of $A_0$ is 2.
\end{proof}

 Let  $h:\Sigma _3^+ (A_0)\rightarrow K$:
\begin{equation*}
h(\alpha )=\bigcap \limits_{n=0}^\infty \overline
{\varphi^{-n}(int\Lambda _{\alpha _n})}, \hspace{0.5cm}\alpha
=(\alpha_n)\in\Sigma _3^+(A_0). \end{equation*} One can see that the
map $h$ is surjective and continuous, and it satisfies $h\circ
\sigma_{A_0} =\varphi \circ h$. That is, $h$ is a topologically
semi-conjugate map for which  the Kasner map is a factor of the
subshift $\sigma_{A_0}:\Sigma_3^+(A_0)\rightarrow\Sigma_3^+(A_0)$.
Therefore, one can easily obtain the following proposition.

\begin{proposition}\label{prop4p4}
 Let $h:\Sigma _3^+(A_0)
\rightarrow K$ be the topologically semi-conjugate map for which the
Kasner map $\varphi$ is a factor of  $\sigma _{A_0}$ as mentioned
above, then one has
\begin{romanlist}[(iii)]
\item For any $y\in K$, $y$ has exactly one or two pre-images
in $\Sigma _3^+(A_0)$, i.e. $h^{-1}(y)$ consists of either one or
two points,
\item The set of points $y \in K$ with $h^{-1}(y)$
consisting of more than one point is contained in the countable set
$\bigcup\limits_{n=1} ^ \infty h^{-n}(\{T_1,T_2,T_3\})$.
\end{romanlist}
\end{proposition}

\begin{remark} As mentioned above,  the Kasner map( actually the BKL
map) represented by the Gauss map has chaotic properties and its
topological entropy is also $\log 2$ even though the Gauss map
itself corresponds to a specific time slicing and ambiguity of time
is manifest in cosmological model area \cite{NCJL1997}. Here we have
proved that the Kasner map represented as in \cite{Ren 1997}, which
doesn't depend on any time slicing, is chaotic in the sense of
Devaney as well as Li-Yorke. Moreover we have found that this map is
a factor of a subshift and its topological entropy is also  $\log
2$. Proposition 4.4 shows clearly the relation between the Kasner
map and the subshift $\sigma_{A_0}$.
\end{remark}

\section{Conclusion}\label{con}
 In this paper,
 the topological entropy for
strictly $A$-coupled-expanding maps on compact metric spaces has
been studied as a criterion for chaos in the sense of Li-Yorke has
been presented, which is less conservative and  more generalized
than the latest result. Some conditions for $A$-coupled-expanding
maps excluding the strictness to be factors of subshifts of finite
type  have been obtained, along with the topological entropy for
$A$-coupled-expanding circle maps excluding the strictness. In
particular, the topological entropy of our proposed
"partition-$A$-coupled-expanding'' map has been estimated and
illustrated with the Kasner map.

\section*{Acknowledgment}
 This work is supported
by the Chinese National Natural Foundation under Grant No. 61174028.
Also it is with immense gratitude that I acknowledge anonymous reviewers.

\end{document}